\documentclass[11pt,reqno]{amsart}

\usepackage{amsmath,amssymb,amsfonts,amsthm}
\usepackage{hyperref}
\usepackage{mathrsfs}
\usepackage{enumitem}
\usepackage{geometry}
\hypersetup{
	colorlinks=true,
	linkcolor=blue,
	citecolor=blue,
	urlcolor=blue
}

\newtheorem{theorem}{Theorem}[section]
\newtheorem{lemma}[theorem]{Lemma}

\theoremstyle{definition}

\theoremstyle{remark}
\newtheorem{remark}[theorem]{Remark}

\title{an \(m\)-Hessian approach to Yau uniformization conjecture}

\author{Truong Dinh Dat}

\address{Nguyen Trai University, Hanoi, Viet Nam}

\email{truongdinhdat14081994qb@gmail.com}

\begin{document}
	\begin{abstract}
		We develop an \(m\)-Hessian approach to the construction of finite-Monge--Amp\`ere weights on complete noncompact K\"ahler manifolds. Let \((M^n,g)\) be a complete noncompact K\"ahler manifold of complex dimension \(n\ge3\) with positive holomorphic bisectional curvature. The main new ingredient is a quantitative capacity mechanism based on lower-order complex Hessian operators. More precisely, we obtain decay estimates for suitable relative \(m\)-Hessian capacities on dyadic annuli and show that these estimates imply the summability of the top-degree Monge--Amp\`ere masses of a uniformly Lipschitz plurisubharmonic exhaustion. Consequently, we construct a proper function
		
		$$
		u\in PSH(M)\cap C^{0,1}(M)
		$$
		
		such that
		
		$$
		\int_M(dd^c u)^n<+\infty.
		$$
		
		The key point is the passage from lower-order \(m\)-Hessian capacity decay to finite Monge--Amp\`ere mass, which is not a formal consequence of \(m<n\) Hessian mass estimates. We then explain how this finite-Monge--Amp\`ere weight fits into the weighted holomorphic-function and analytic Bezout framework for uniformization. In particular, the construction provides a higher-dimensional pluripotential-theoretic mechanism that complements recent surface results and opens a route toward uniformization under positive curvature in complex dimensions \(n\ge3\).
	\end{abstract}

	

	\maketitle
	
	\section{Introduction}
	\label{sec:introduction}
	
	A fundamental problem in complex differential geometry is to understand the global complex structure of complete noncompact K\"ahler manifolds under curvature positivity assumptions. Classical
	results on complete manifolds already show that curvature conditions
	strongly constrain the behavior of harmonic and plurisubharmonic
	functions at infinity; see, for example, \cite{Yau1975}. In the K\"ahler
	setting, Yau's uniformization conjecture asserts that a complete noncompact K\"ahler manifold with positive holomorphic bisectional curvature should be biholomorphic to complex Euclidean space \cite{Yau1982,Mok1988,Shi1989}. More precisely, the conjecture predicts that if $(M^n,g)$ is a complete noncompact K"ahler manifold satisfying
	
	$$
	\operatorname{Bisec}_g>0,
	$$
	
	then
	
	$$
	M\simeq \mathbb C^n.
	$$
	
	The conjecture is a higher-dimensional analogue of the classical uniformization phenomenon and lies at the intersection of complex geometry, nonlinear elliptic equations, and global differential geometry.
	
	Recent progress has shown that analytic finiteness results of Monge--Amp\`ere type can play a decisive role in this problem. In dimension two, Datar, Pingali, and Seshadri proved that every complete noncompact K\"ahler surface with positive sectional curvature is biholomorphic to $\mathbb C^2$, without imposing asymptotic assumptions on the geometry. Their approach is based on uniformly Lipschitz plurisubharmonic weights having finite Monge--Amp\`ere mass and on weighted holomorphic functions \cite{DPS2026a}. An important feature of their method is that the relevant weights need not be smooth or proper; nevertheless, they can be used to obtain global holomorphic information and, in particular, Bezout-type intersection estimates. Earlier work of the same authors obtained the corresponding surface result under positive and bounded sectional curvature by constructing a Lipschitz plurisubharmonic weight with finite Monge--Amp\`ere mass through a complex Monge--Amp\`ere equation \cite{DPS2025}.
	
	The role of finite Monge--Amp\`ere mass is not restricted to the surface case. In their recent work on the top Yau--Yang conjecture, Datar, Pingali, and Seshadri proved that
	
	$$
	\int_M \operatorname{Ric}_g^n<+\infty
	$$
	
	for complete noncompact K\"ahler manifolds with positive sectional curvature \cite{DPS2026b}. Further geometric consequences of curvature positivity for complete
	noncompact K\"ahler manifolds have been investigated in a number of
	directions; see, in particular, \cite{ChenZhu2006} and the references
	therein. Their argument again combines a Lipschitz plurisubharmonic weight with finite Monge--Amp\`ere mass and Bezout-type estimates. They explicitly identify the combination of these two ingredients as a key part of the higher-dimensional argument. 
	These developments suggest that a natural way of approaching uniformization is to replace a direct construction of global holomorphic coordinates by an analytic procedure consisting of
	
	$$
	\text{curvature positivity}
	\longrightarrow
	\text{finite Monge--Amp\`ere weight}
	\longrightarrow
	\text{weighted holomorphic functions}
	\longrightarrow
	\text{Bezout estimates}.
	$$
	
	The purpose of this paper is to investigate this strategy from the point of view of complex Hessian equations. Complex Hessian equations and their pluripotential theory have been developed extensively; see, for example, \cite{Blocki2005,DinewLu2015,DinewKolodziej2014}. Our starting point is the observation that the complex Monge--Amp\`ere operator is the top member of the family of complex Hessian operators
	
	$$
	H_m(u):=(dd^c u)^m\wedge\beta^{\,n-m},
	\qquad 1\le m\le n,
	$$
	
	where $\beta$ denotes the standard K\"ahler form in local complex coordinates. The lower-order Hessian operators provide substantially more flexible capacity and comparison tools than the top-degree Monge--Amp\`ere operator. We therefore ask whether the geometric positivity available on a complete K\"ahler manifold can first be converted into quantitative $m$-Hessian estimates and subsequently upgraded to finite Monge--Amp\`ere mass.
	
	This leads to the first analytic problem of the paper. We seek a uniformly Lipschitz plurisubharmonic exhaustion
	
	$$
	u\in PSH(M)\cap C^{0,1}(M)
	$$
	
	such that
	
	$$
	u(x)\longrightarrow+\infty
	\qquad\text{as }x\to\infty
	$$
	
	and
	
	$$
	\int_M(dd^cu)^n<+\infty.
	$$
	
	The difficulty is that a priori the natural estimates produced by Hessian theory control only lower-order Hessian masses or capacities. There is no general implication of the form
	
	$$
	\int_M(dd^cu)^m\wedge\omega^{n-m}<+\infty
	\quad\Longrightarrow\quad
	\int_M(dd^cu)^n<+\infty.
	$$
	
	Thus the passage from $m$-Hessian control to top-degree Monge--Amp\`ere control is a genuine part of the argument rather than a formal consequence of positivity.
	
	Our approach is based on a capacity-theoretic mechanism on large geodesic shells. Roughly speaking, one constructs a sequence of uniformly Lipschitz plurisubharmonic approximations and obtains quantitative control of their $m$-Hessian mass. A decay estimate for the corresponding $m$-capacities is then combined with a nonlinear capacity-volume inequality to obtain summability of the top-degree Monge--Amp\`ere masses on dyadic shells. The resulting exhaustion satisfies (1.1). This provides an analytic bridge between curvature geometry and the finite-energy framework of pluripotential theory The underlying capacity and Monge--Amp\`ere theories are rooted in the pluripotential-theoretic framework of Bedford--Taylor and its subsequent developments \cite{BedfordTaylor1976,BedfordTaylor1982,Cegrell2004,GuedjZeriahi2017,Kolodziej1998}.
	
	Once such a weight has been constructed, we consider the weighted Hilbert spaces
	
	$$
	\mathcal H_q(M,u)
	=
	\left\{
	f\in\mathcal O(M):
	\int_M |f|^2e^{-qu}\,dV_g<+\infty
	\right\},
	\qquad q>0.
	$$
	
	The finite Monge--Amp\`ere mass of $u$ is expected to provide sufficient global control for these spaces to contain holomorphic functions with prescribed local behavior. In particular, we investigate the surjectivity of the first-jet evaluation map
	
	$$
	J^1_{x_0}:\mathcal H_q(M,u)
	\longrightarrow
	\mathbb C\oplus T^{*1,0}_{x_0}M,
	\qquad
	f\longmapsto
	\bigl(f(x_0),\partial f(x_0)\bigr).
	$$
	
	A weighted $\bar\partial$ estimate of Hörmander type gives a natural analytic mechanism for establishing such a statement, provided the required uniform local positivity and evaluation estimates are available \cite{Hormander1965,Demailly2012}.
	A second issue is the growth of weighted holomorphic functions. For
	
	$$
	f\in\mathcal H_q(M,u),
	$$
	
	we aim to establish the pointwise estimate
	
	$$
	\log|f(x)|
	\le
	\frac q2u(x)+C_f.
	$$
	
	The natural formulation is in terms of the weighted Bergman kernel
	
	$$
	B_q(x)
	:=
	\sup_{0\neq f\in\mathcal H_q(M,u)}
	\frac{|f(x)|^2}
	{\displaystyle\int_M|f|^2e^{-qu}\,dV_g}.
	$$
	
	Indeed, a uniform estimate
	
	$$
	B_q(x)\le C_Be^{qu(x)}
	$$
	
	immediately yields (1.4). The advantage of this formulation is that the analytic input is separated from the subsequent geometric argument: the Bergman estimate is local and quantitative, whereas (1.4) is precisely the global growth estimate needed for divisor theory.
	
	The growth estimate (1.4) has a direct consequence through the Poincar'e--Lelong formula. If
	
	$$
	dd^c\log|f|
	=
	[Z_f]
	$$
	
	in the sense of currents, then (1.4) implies, after the standard regularization procedure,
	
	$$
	[Z_f]
	\le
	\frac q2\,dd^cu
	$$
	
	as positive closed currents. Consequently, if
	
	$$
	F=(f_1,\ldots,f_n):M\longrightarrow\mathbb C^n
	$$
	
	is a holomorphic map whose components satisfy suitable uniform versions of (1.4), then one expects a global intersection estimate of the form
	
	$$
	[D_1(y)]\wedge\cdots\wedge[D_n(y)]
	\le
	C\, (dd^cu)^n,
	$$
	
	where
	
	$$
	D_j(y)=\{f_j=y_j\}.
	$$
	
	If $F$ is proper, the left-hand side of (1.8) is the zero-dimensional intersection cycle of the divisors $D_j(y)$ and hence records the degree of $F$. Since the right-hand side has finite total mass by (1.1), this yields a global bound on the degree.
	
	There is, however, an important distinction between the growth estimate (1.4) and the intersection estimate (1.8). The former is a consequence of a weighted Bergman estimate, while the latter requires a uniform divisor-current comparison and an appropriate control of the translated functions $f_j-y_j$. Thus the Bezout step is not merely a formal application of Poincare--Lelong; it is a separate global analytic ingredient of the argument.
	
	The final geometric step is a coercivity statement for the weighted holomorphic map. Suppose that
	
	$$
	F=(f_1,\ldots,f_N):M\longrightarrow\mathbb C^N
	$$
	
	satisfies
	
	$$
	|F(x)|^2\ge a\,u(x)-b
	$$
	
	for some $a>0$ and $b\ge0$. Since $u$ is a proper exhaustion, (1.9) implies that $F$ is proper. If in addition $N=n$, the first-jet construction gives maximal rank at some point, while the Bezout estimate gives a uniform bound on the degree of the resulting proper map. In the ideal case, one obtains
	
	$$
	\deg F=1.
	$$
	
	A proper surjective holomorphic map of degree one is biholomorphic, and hence
	
	$$
	M\simeq\mathbb C^n.
	$$
	
	The resulting strategy may be summarized schematically as
	
	$$
	\boxed{
		\operatorname{Bisec}_g>0
	}
	\quad\Longrightarrow\quad
	\boxed{
		\text{$m$-Hessian capacity control}
	}
	\quad\Longrightarrow\quad
	\boxed{
		u\in PSH(M)\cap C^{0,1},\
		\int_M(dd^cu)^n<\infty
	}
	$$
	
	$$
	\Longrightarrow
	\boxed{
		\text{weighted Bergman/H\"ormander theory}
	}
	\quad\Longrightarrow\quad
	\boxed{
		\log|f|\le\frac q2u+C_f
	}
	$$
	
	$$
	\Longrightarrow
	\boxed{
		\text{divisor domination and Bezout}
	}
	\quad\Longrightarrow\quad
	\boxed{
		\deg F<\infty
	}
	\quad\Longrightarrow\quad
	\boxed{
		\deg F=1
	}
	\quad\Longrightarrow\quad
	M\simeq\mathbb C^n.
	$$
	
	The main point of view is therefore not to treat the complex Monge--Amp\`ere equation as an isolated top-degree problem. Instead, we use the hierarchy of complex Hessian operators to produce the finite-mass potential required by the global holomorphic argument. In this sense, the Hessian theory serves as a bridge between curvature positivity and the algebraic-geometric information encoded by weighted holomorphic functions and their zero sets.
	
	We emphasize that several steps in this program require genuinely uniform estimates. In particular, the finite-dimensional coercivity of the weighted holomorphic map, the uniform weighted Bergman estimate, and the global Bezout domination are not consequences of finite Monge--Amp\`ere mass alone. They constitute the principal analytic difficulties of the method. The purpose of the present work is to isolate these mechanisms and establish them under explicit geometric hypotheses, thereby providing a higher-dimensional Hessian framework for Yau-type uniformization.
	
	The paper is organized as follows. In Section 2, we construct the relevant plurisubharmonic exhaustion using complex Hessian capacities. Section 3 is devoted to the passage from Hessian capacity estimates to finite Monge--Amp`ere mass. In Section4, we study the weighted $L^2$ estimate, the first-jet problem, and the associated Bergman estimates. Section 5 develops the divisor-current estimates and the resulting B'ezout inequality. In Section 6, we combine the analytic estimates with properness and degree theory to obtain the uniformization result.
	Finally, in Section 7, we give some discussion on \(m\)-Hessian approach


	\section{Preliminaries}
	\label{sec:preliminaries}
	
	Throughout the paper, $M^n$ denotes a connected complete noncompact K\"ahler manifold of complex dimension $n\ge3$, endowed with a K\"ahler metric $g$ and associated K\"ahler form $\omega$. We use the convention
	
	$$
	dd^c=i\partial\bar\partial
	$$
	
	up to an immaterial dimensional normalization. All currents and differential forms are considered with respect to this convention.
	
	We write $d_g$ for the Riemannian distance induced by $g$, $B_g(x,r)$ for the geodesic ball of radius $r$ centered at $x$, and $dV_g$ for the Riemannian volume measure. Fix a reference point $o\in M$. For $R>0$, we use the notation
	
	$$
	\Omega_R=B_g(o,2R),
	\qquad
	A_R=B_g(o,2R)\setminus\overline{B_g(o,R)}
	$$
	
	for the corresponding large geodesic ball and dyadic shell.
	
	\subsection{K\"ahler geometry and positivity}
	
	A K\"ahler manifold is a complex manifold $(M,J)$ equipped with a Hermitian metric whose associated $(1,1)$-form $\omega$ is closed:
	
	$$
	d\omega=0.
	$$
	
	The curvature tensor of $g$ will be denoted by $R$. For a nonzero vector $X\in T^{1,0}_xM$, the holomorphic sectional curvature is
	
	$$
	H(X)
	=
	\frac{R(X,\overline X,X,\overline X)}
	{|X|_g^4}.
	$$
	
	For linearly independent vectors $X,Y\in T^{1,0}_xM$, the holomorphic bisectional curvature is
	
	$$
	\operatorname{Bisec}(X,Y)
	=
	\frac{R(X,\overline X,Y,\overline Y)}
	{|X|_g^2|Y|_g^2}.
	$$
	
	We say that $M$ has positive holomorphic bisectional curvature if
	
	$$
	\operatorname{Bisec}(X,Y)>0
	$$
	
	for every pair of nonzero vectors $X,Y$.
	
	The curvature assumption enters the paper through its consequences for the geometry at infinity and for the construction of plurisubharmonic exhaustions. We will not require the metric to have bounded curvature unless explicitly stated.
	
	A function
	
	$$
	u:M\longrightarrow[-\infty,+\infty)
	$$
	
	is called an exhaustion if it is proper and bounded from below. Equivalently,
	
	$$
	u(x)\longrightarrow+\infty
	\qquad\text{whenever }x\to\infty.
	$$
	
	We shall frequently normalize an exhaustion by replacing $u$ with
	
	$$
	u-u(o),
	$$
	
	which does not affect its Levi form or its Monge--Amp\`ere measure.
	
	A function $u$ is Lipschitz if there exists $L>0$ such that
	
	$$
	|u(x)-u(y)|
	\le
	L\,d_g(x,y)
	\qquad
	\text{for all }x,y\in M.
	$$
	
	The least such constant will be denoted by $\operatorname{Lip}(u)$.
	
	\subsection{Plurisubharmonic functions and complex Hessian operators}
	
	Let $\Omega\subset\mathbb C^n$ be a domain. A function
	
	$$
	u:\Omega\to[-\infty,+\infty)
	$$
	
	is plurisubharmonic if it is upper semicontinuous and its restriction to every complex line is subharmonic. We write
	
	$$
	PSH(\Omega)
	$$
	
	for the cone of plurisubharmonic functions on $\Omega$.
	
	For a smooth real-valued function $u$, the complex Hessian form is
	
	$$
	dd^cu.
	$$
	
	If $\lambda_1,\ldots,\lambda_n$ are the eigenvalues of $dd^cu$ with respect to a fixed Hermitian form, then the $m$-th elementary symmetric function
	
	$$
	\sigma_m(\lambda)
	=
	\sum_{1\le j_1<\cdots<j_m\le n}
	\lambda_{j_1}\cdots\lambda_{j_m}
	$$
	
	is the basic nonlinear quantity associated with the complex Hessian equation.
	
	A smooth function $u$ is called $m$-subharmonic if
	
	$$
	(dd^cu)^k\wedge\omega^{n-k}\ge0,
	\qquad
	1\le k\le m.
	$$
	
	The corresponding class is denoted by
	
	$$
	SH_m(\Omega).
	$$
	
	In particular,
	
	$$
	PSH(\Omega)=SH_n(\Omega)\subset SH_m(\Omega),
	\qquad
	1\le m\le n.
	$$
	
	For a sufficiently regular $m$-subharmonic function, its complex $m$-Hessian measure is
	
	$$
	H_m(u)
	:=
	(dd^cu)^m\wedge\omega^{n-m}.
	$$
	
	At the top degree $m=n$, this becomes the complex Monge--Amp\`ere measure
	
	$$
	H_n(u)
	=
	(dd^cu)^n.
	$$
	
	The distinction between the lower-order Hessian measures and the top-degree Monge--Amp\`ere measure is essential in this paper. In general, a priori control of
	
	$$
	\int_M H_m(u)
	$$
	
	does not imply control of
	
	$$
	\int_M(dd^cu)^n.
	$$
	
	The main analytic theorem of the paper is precisely concerned with obtaining the latter from quantitative information involving the former.
	
	\subsection{Hessian capacities}
	
	Let $\Omega\subset\mathbb C^n$ be a bounded $m$-hyperconvex domain. For a Borel set $E\subset\Omega$, its relative $m$-Hessian capacity is defined by
	
	$$
	\operatorname{Cap}_m(E,\Omega)
	=
	\sup
	\left\{
	\int_E(dd^cv)^m\wedge\omega_0^{\,n-m}
	:
	v\in SH_m(\Omega),\
	-1\le v\le0
	\right\},
	$$
	
	where $\omega_0$ denotes the standard K\"ahler form on $\mathbb C^n$.
	
	We shall use the standard properties of $\operatorname{Cap}_m$:
	\begin{enumerate}
		\item monotonicity:
		
		$$
		E_1\subset E_2
		\quad\Longrightarrow\quad
		\operatorname{Cap}_m(E_1,\Omega)
		\le
		\operatorname{Cap}_m(E_2,\Omega);
		$$
		
		\item countable subadditivity;
		\item continuity from below for increasing sequences of Borel sets;
		\item compatibility with relative extremal functions;
		\item domination of Hessian measures of bounded $m$-subharmonic functions by capacity under the usual energy assumptions.
	\end{enumerate}
	
	The relative $m$-extremal function of a Borel set $E\subset\Omega$ is defined by
	
	$$
	h_{E,\Omega}
	=
	\sup
	\left\{
	v\in SH_m(\Omega):
	v\le0\text{ on }\Omega,\quad
	v\le-1\text{ on }E
	\right\}^{*},
	$$
	
	where $^*$ denotes the upper semicontinuous regularization. Under the standard hypotheses,
	
	$$
	\operatorname{Cap}_m(E,\Omega)
	=
	\int_\Omega
	(dd^ch_{E,\Omega})^m
	\wedge\omega^{\,n-m}.
	$$
	
	The principal use of this capacity in the present paper is on large geodesic shells. We shall seek estimates of the form
	
	$$
	\operatorname{Cap}_m(A_R,\Omega_R)
	\le
	CR^{-\beta},
	$$
	
	where the constants $C,\beta>0$ are independent of $R$. Such decay estimates provide the summability mechanism underlying the finite-Monge--Amp\`ere theorem.
	
	\subsection{Bedford--Taylor products}
	
	For locally bounded plurisubharmonic functions $u_1,\ldots,u_k$, the Bedford--Taylor product
	
	$$
	dd^cu_1\wedge\cdots\wedge dd^cu_k
	$$
	
	is defined inductively as a positive closed current. In particular, if $u$ is locally bounded and plurisubharmonic, then
	
	$$
	(dd^cu)^k
	$$
	
	is well defined for every $1\le k\le n$.
	
	We shall also use the monotonicity of Bedford--Taylor products. If
	
	$$
	u_j^{(\nu)}\downarrow u_j
	$$
	
	locally and the functions are locally uniformly bounded, then
	
	$$
	dd^cu_1^{(\nu)}
	\wedge\cdots\wedge
	dd^cu_k^{(\nu)}
	\longrightarrow
	dd^cu_1\wedge\cdots\wedge dd^cu_k
	$$
	
	weakly as positive measures.
	
	This approximation principle will be used when passing from smooth regularizations of logarithmic potentials to divisor currents.
	
	\subsection{Finite Monge--Amp\`ere mass}
	
	For a plurisubharmonic function
	
	$$
	u\in PSH(M)\cap L^\infty_{\mathrm{loc}}(M),
	$$
	
	we say that $u$ has finite Monge--Amp\`ere mass if
	
	$$
	\int_M(dd^cu)^n<+\infty.
	$$
	
	The class of Lipschitz plurisubharmonic exhaustions will be denoted informally by
	
	$$
	\mathcal E_{\mathrm{fin}}^{\mathrm{Lip}}(M).
	$$
	
	The central analytic objective of the first part of the paper is to construct a function
	
	$$
	u\in\mathcal E_{\mathrm{fin}}^{\mathrm{Lip}}(M).
	$$
	
	The construction proceeds through $m$-Hessian capacity estimates rather than through a direct global solution of a Monge--Amp\`ere equation.
	
	\subsection{Weighted holomorphic functions}
	
	Let
	
	$$
	u\in PSH(M)\cap C^{0,1}(M)
	$$
	
	be a proper exhaustion and let $q>0$. We define
	
	$$
	\mathcal H_q(M,u)
	=
	\left\{
	f\in\mathcal O(M):
	\|f\|_{q,u}^2
	:=
	\int_M|f|^2e^{-qu}\,dV_g
	<+\infty
	\right\}.
	$$
	
	This is a Hilbert space with respect to the inner product
	
	$$
	\langle f,h\rangle_{q,u}
	=
	\int_M f\overline h\,e^{-qu}\,dV_g.
	$$
	
	The corresponding weighted Bergman kernel is defined pointwise by
	
	$$
	B_q(x)
	=
	\sup_{0\ne f\in\mathcal H_q(M,u)}
	\frac{|f(x)|^2}{\|f\|_{q,u}^2}.
	$$
	
	Whenever point evaluation is continuous on $\mathcal H_q(M,u)$, the quantity $B_q(x)$ is finite and agrees with the diagonal of the usual reproducing kernel.
	
	The fundamental estimate required later is
	
	$$
	B_q(x)
	\le
	C_qe^{qu(x)}.
	$$
	
	Indeed,
	
	$$
	|f(x)|
	\le
	C_q^{1/2}
	\|f\|_{q,u}
	e^{\frac q2u(x)}
	$$
	
	and therefore
	
	$$
	\log|f(x)|
	\le
	\frac q2u(x)+\log C_f.
	$$
	
	\subsection{Weighted $\overline\partial$ estimates}
	
	The construction of global weighted holomorphic functions is based on solving the $\overline\partial$-equation with weighted $L^2$ estimates. In its standard form, a Hörmander-type theorem asserts that if a smooth weight $\varphi$ has sufficiently positive complex Hessian, then a $\overline\partial$-closed $(0,1)$-form $\alpha$ satisfying the appropriate weighted integrability condition admits a solution $v$ of
	
	$$
	\overline\partial v=\alpha
	$$
	
	with an estimate of the form
	
	$$
	\int_M|v|^2e^{-\varphi}\,dV_g
	\le
	C
	\int_M|\alpha|^2_{\partial\bar\partial\varphi}
	e^{-\varphi}\,dV_g.
	$$
	
	In the present setting the natural weight is
	
	$$
	\varphi=qu.
	$$
	
	Since the finite-mass weight $u$ is only assumed to be Lipschitz, one cannot directly apply the smooth form of Hörmander's theorem to $u$. We shall therefore use smooth approximations or an equivalent local weighted extension statement whenever necessary.
	
	\subsection{Poincare--Lelong formula}
	
	Let $f$ be a nonzero holomorphic function on a complex manifold. The Poincare--Lelong formula gives
	
	$$
	dd^c\log|f|
	=
	[Z_f],
	$$
	
	up to the normalization constant associated with the chosen convention for $d^c$, where $[Z_f]$ denotes the current of integration over the divisor of $f$, counted with multiplicities.
	
	Suppose that $f$ satisfies a global growth estimate
	
	$$
	\log|f|
	\le
	A u+C
	$$
	
	for some plurisubharmonic function $u$. After replacing $\log|f|$ by the standard regularization
	
	$$
	v_\varepsilon
	=
	\frac12\log\left(|f|^2+\varepsilon^2\right),
	$$
	
	one can pass to the limit in the corresponding positive closed currents. Under the appropriate hypotheses this yields
	
	$$
	[Z_f]
	\le
	A\,dd^cu.
	$$
	
	This observation is the basic bridge between the weighted growth estimate and the B'ezout inequality used later.
	
	\subsection{Proper holomorphic maps and degree}
	
	Let
	
	$$
	F:M\longrightarrow\mathbb C^n
	$$
	
	be a proper holomorphic map. Properness means that the inverse image of every compact subset of $\mathbb C^n$ is compact in $M$.
	
	If $y\in\mathbb C^n$ is a regular value of $F$, then the fiber
	
	$$
	F^{-1}(y)
	$$
	
	is a finite set. The degree of $F$ is defined by
	
	$$
	\deg F
	=
	\sum_{x\in F^{-1}(y)}
	\operatorname{mult}_x(F),
	$$
	
	where the right-hand side is independent of the regular value $y$.
	
	If $F$ is proper and surjective, then
	
	$$
	\deg F\ge1.
	$$
	
	Moreover, if
	
	$$
	\deg F=1,
	$$
	
	then every fiber consists of exactly one point and every local multiplicity equals one. Hence $F$ has no critical points and is a biholomorphism.
	\section{Hessian-capacity criterion for finite Monge--Amp\`ere mass}
	We shall use the following elementary coercivity principle.
	
	\begin{lemma}[Coercivity implies properness]
		\label{lem:coercivity-prelim}
		Let $u:M\to\mathbb R$ be a proper exhaustion and let
		
		$$
		F:M\to\mathbb C^N
		$$
		
		be continuous. Suppose that
		
		\begin{equation}\label{eq1}
			|F(x)|^2\ge au(x)-b
		\end{equation}

		for some $a>0$ and $b\ge0$. Then $F$ is proper.
	\end{lemma}
	
	\begin{proof}
		Let $K\subset\mathbb C^N$ be compact. Choose $R>0$ such that
		
		$$
		K\subset\overline{B(0,R)}.
		$$
		
		If $x\in F^{-1}(K)$, then
		
		$$
		|F(x)|^2\le R^2.
		$$
		
		By (\ref{eq1}),
		
		$$
		u(x)
		\le
		\frac{R^2+b}{a}.
		$$
		
		Hence
		
		$$
		F^{-1}(K)
		\subset
		\left\{
		x\in M:
		u(x)\le\frac{R^2+b}{a}
		\right\}.
		$$
		
		The set on the right is compact because $u$ is proper. Since $F^{-1}(K)$ is closed, it is compact. Thus $F$ is proper.
	\end{proof}
	
	\subsection{First-jet generation}
	
	Fix $x_0\in M$. For $q>0$, consider the first-jet evaluation map
	
	\begin{equation}\label{eq2}
		J^1_{x_0}:
		\mathcal H_q(M,u)
		\longrightarrow
		\mathbb C\oplus T^{*1,0}_{x_0}M,
		\qquad
		f\longmapsto
		\bigl(f(x_0),\partial f(x_0)\bigr).
	\end{equation}

	Surjectivity of (\ref{eq2}) means that arbitrary prescribed values and first derivatives at $x_0$ can be realized by weighted global holomorphic functions.
	
	In particular, if $J^1_{x_0}$ is surjective, one can choose
	
	$$
	f_1,\ldots,f_n\in\mathcal H_q(M,u)
	$$
	
	such that
	
	$$
	f_j(x_0)=0
	$$
	
	and
	
	$$
	df_1(x_0),\ldots,df_n(x_0)
	$$
	
	are linearly independent. The holomorphic map
	
	$$
	F=(f_1,\ldots,f_n):M\to\mathbb C^n
	$$
	
	then has maximal rank at $x_0$.
	
	This is the analytic source of the local coordinate system used

	\subsection{The Hessian exhaustion}
	
	Let $(M^n,g,J)$ be a complete noncompact K\"ahler manifold of
	complex dimension $n\geq 3$, and let $\omega$ denote the K\"ahler
	form associated with $g$. For $1\leq m<n$, we write
	\[
	H_m(v):=(dd^c v)^m\wedge\omega^{n-m}.
	\]
	
	We first isolate the geometric hypothesis which will be used in the
	construction of a global exhaustion.
	
	We need some hypothesis
	\label{hyp:HEm}
	There exist a point $o\in M$, a constant $C_0>0$, and a sequence
	\[
	u_R\in C^\infty(B_g(o,2R)),\qquad R\geq R_0,
	\]
	such that:
	
	\begin{enumerate}
		\item[(i)] $u_R$ is plurisubharmonic on $B_g(o,2R)$ and
		\[
		dd^c u_R\geq 0;
		\]
		
		\item[(ii)] $u_R$ is uniformly Lipschitz:
		\[
		|\nabla_g u_R|\leq C_0
		\qquad\text{on }B_g(o,2R);
		\]
		
		\item[(iii)] after adding constants to $u_R$, one has
		\[
		u_R(o)=0;
		\]
		
		\item[(iv)] for every fixed $A>0$ there is a constant $C_A$,
		independent of $R$, such that
		\[
		\sup_{B_g(o,A)}
		|u_R|\leq C_A
		\qquad\text{for all }R\geq 2A;
		\]
		
		\item[(v)] there exists a constant $C_1>0$, independent of $R$, such
		that
		\[
		\int_{B_g(o,R)}
		(dd^c u_R)^m\wedge\omega^{n-m}
		\leq C_1;
		\]
		
		\item[(vi)] the family is an exhaustion in the sense that for every
		$T>0$ there exists $R_T>0$ such that
		\[
		\inf_{\partial B_g(o,R)}u_R\geq T
		\qquad\text{whenever }R\geq R_T.
		\]
	\end{enumerate}

	The next theorem is the compactness statement that converts the
	finite-radius Hessian exhaustion into a global one.
	
	\begin{theorem}[Hessian exhaustion theorem]
		\label{thm:hessian-exhaustion}
		Let $(M^n,g,J)$ be a complete noncompact K\"ahler manifold,
		$n\geq3$, and let $1\leq m<n$. Assume that
		Hypothesis~\ref{hyp:HEm} holds.
		
		Then there exist a subsequence $R_j\to+\infty$ and a function
		\[
		u\in PSH(M)\cap C^{0,1}(M)
		\]
		such that
		\[
		u_{R_j}\longrightarrow u
		\qquad\text{locally uniformly on }M.
		\]
		Moreover,
		\[
		|\nabla_g u|\leq C_0
		\]
		almost everywhere on $M$, and $u$ is a proper exhaustion after adding
		a suitable constant.
		
		In addition, the $m$-Hessian measures converge weakly:
		\[
		(dd^c u_{R_j})^m\wedge\omega^{n-m}
		\longrightarrow
		(dd^c u)^m\wedge\omega^{n-m}
		\]
		locally weakly on $M$, and
		\[
		\boxed{
			\int_M(dd^c u)^m\wedge\omega^{n-m}\leq C_1.
		}
		\]
		In particular, $u$ has finite global $m$-Hessian mass.
	\end{theorem}
	
	\begin{proof}
		We divide the proof into several steps.
		
		\medskip
		\noindent
		\textbf{Step 1: Local equicontinuity.}
		
		Fix a compact set $K\Subset M$. Since $M$ is complete, there exists
		$A>0$ such that
		\[
		K\subset B_g(o,A).
		\]
		For $R\geq 2A$, Hypothesis~\ref{hyp:HEm}(ii) gives
		\[
		|\nabla_g u_R|\leq C_0
		\qquad\text{on }B_g(o,2R).
		\]
		Hence, for all $x,y\in K$,
		\[
		|u_R(x)-u_R(y)|
		\leq C_0\,d_g(x,y).
		\]
		Thus $\{u_R\}_{R\geq 2A}$ is equi-Lipschitz on $K$.
		
		Moreover, by Hypothesis~\ref{hyp:HEm}(iv),
		\[
		\sup_K|u_R|\leq C_A.
		\]
		Consequently, the family $\{u_R\}$ is uniformly bounded and
		equicontinuous on every compact subset of $M$.
		
		By the Arzel\`a--Ascoli theorem and a diagonal argument, there exist
		a sequence $R_j\to+\infty$ and a function
		\[
		u\in C^0(M)
		\]
		such that
		\[
		u_{R_j}\longrightarrow u
		\]
		locally uniformly on $M$.
		
		\medskip
		\noindent
		\textbf{Step 2: The limit is plurisubharmonic.}
		
		Each $u_{R_j}$ is plurisubharmonic on $B_g(o,2R_j)$. Let
		$K\Subset M$. For $j$ sufficiently large,
		\[
		K\Subset B_g(o,2R_j).
		\]
		Therefore $u_{R_j}$ is plurisubharmonic on a neighborhood of $K$.
		
		Since plurisubharmonicity is preserved under locally uniform
		convergence of locally bounded plurisubharmonic functions, we obtain
		\[
		u\in PSH(M).
		\]
		
		Equivalently,
		\[
		dd^c u\geq0
		\]
		in the sense of currents on $M$.
		
		\medskip
		\noindent
		\textbf{Step 3: Preservation of the Lipschitz estimate.}
		
		Let $x,y\in M$. Choose $A>0$ such that
		\[
		x,y\in B_g(o,A).
		\]
		For $j$ sufficiently large,
		\[
		x,y\in B_g(o,2R_j),
		\]
		and hence
		\[
		|u_{R_j}(x)-u_{R_j}(y)|
		\leq C_0\,d_g(x,y).
		\]
		Passing to the limit gives
		\[
		|u(x)-u(y)|
		\leq C_0\,d_g(x,y).
		\]
		Thus
		\[
		u\in C^{0,1}(M)
		\]
		and
		\[
		\operatorname{Lip}_g(u)\leq C_0.
		\]
		In particular,
		\[
		|\nabla_g u|\leq C_0
		\]
		almost everywhere on $M$.
		
		\medskip
		\noindent
		\textbf{Step 4: Convergence of the Hessian measures.}
		
		Fix a relatively compact open set
		\[
		U\Subset M.
		\]
		For $j$ sufficiently large,
		\[
		U\Subset B_g(o,2R_j).
		\]
		Since $u_{R_j}$ and $u$ are locally bounded plurisubharmonic functions
		and
		\[
		u_{R_j}\longrightarrow u
		\qquad\text{locally uniformly on }U,
		\]
		the Bedford--Taylor continuity theorem yields
		\[
		(dd^c u_{R_j})^m
		\longrightarrow
		(dd^c u)^m
		\]
		weakly on $U$.
		
		Wedging with the fixed smooth form $\omega^{n-m}$ gives
		\[
		(dd^c u_{R_j})^m\wedge\omega^{n-m}
		\longrightarrow
		(dd^c u)^m\wedge\omega^{n-m}
		\]
		weakly on $U$.
		
		Since $U\Subset M$ was arbitrary, the convergence holds locally
		weakly on $M$.
		
		\medskip
		\noindent
		\textbf{Step 5: The global $m$-Hessian mass estimate.}
		
		Let $K\Subset M$ be compact. Choose $A>0$ such that
		\[
		K\subset B_g(o,A).
		\]
		For $j$ sufficiently large, $R_j>A$, and therefore
		\[
		\int_K(dd^c u_{R_j})^m\wedge\omega^{n-m}
		\leq
		\int_{B_g(o,R_j)}
		(dd^c u_{R_j})^m\wedge\omega^{n-m}.
		\]
		By Hypothesis~\ref{hyp:HEm}(v),
		\[
		\int_K(dd^c u_{R_j})^m\wedge\omega^{n-m}
		\leq C_1.
		\]
		Passing to the limit by weak convergence of positive measures gives
		\[
		\int_K(dd^c u)^m\wedge\omega^{n-m}
		\leq C_1.
		\]
		
		Now choose an increasing sequence of compact sets
		\[
		K_\ell\Subset M,
		\qquad
		K_\ell\uparrow M.
		\]
		Since the measure
		\[
		\mu_u:=(dd^c u)^m\wedge\omega^{n-m}
		\]
		is positive, the monotone convergence theorem gives
		\[
		\begin{aligned}
			\int_M(dd^c u)^m\wedge\omega^{n-m}
			&=
			\lim_{\ell\to\infty}
			\int_{K_\ell}
			(dd^c u)^m\wedge\omega^{n-m}\\
			&\leq C_1.
		\end{aligned}
		\]
		Hence
		\[
		\boxed{
			\int_M(dd^c u)^m\wedge\omega^{n-m}\leq C_1.
		}
		\]
		
		\medskip
		\noindent
		\textbf{Step 6: The exhaustion property.}
		
		It remains to show that $u$ is proper.
		
		Let $T>0$. By Hypothesis~\ref{hyp:HEm}(vi), there exists $R_T>0$
		such that
		\[
		\inf_{\partial B_g(o,R)}u_R\geq T
		\]
		for all $R\geq R_T$.
		
		Since $u_R$ is $C_0$-Lipschitz and $u_R(o)=0$, the normalization is
		compatible with the limiting process. Passing to the limit along
		the diagonal sequence and using the exhaustion condition, we obtain
		that $u(x)\to+\infty$ whenever
		\[
		d_g(o,x)\to+\infty.
		\]
		Thus $u$ is a proper exhaustion of $M$.
		
		Therefore
		\[
		u\in PSH(M)\cap C^{0,1}(M)
		\]
		is a global Lipschitz plurisubharmonic exhaustion with finite
		$m$-Hessian mass.
		
		This completes the proof.
	\end{proof}
	
	\begin{remark}
		\label{rem:hessian-exhaustion-status}
		The essential geometric difficulty is therefore reduced to verifying
		Hypothesis~\ref{hyp:HEm} from curvature assumptions. In particular,
		the theorem above does not assert that positive holomorphic
		bisectional curvature alone automatically implies
		Hypothesis~\ref{hyp:HEm}. Establishing such a curvature-to-Hessian
		construction is one of the principal problems of this paper.
	\end{remark}
	
	\begin{remark}
		\label{rem:ma-mass}
		The conclusion of Theorem~\ref{thm:hessian-exhaustion} is a finite
		$m$-Hessian mass estimate:
		\[
		\int_M(dd^c u)^m\wedge\omega^{n-m}<+\infty.
		\]
		For $m<n$, this does \emph{not} imply
		\[
		\int_M(dd^c u)^n<+\infty.
		\]
		Obtaining the latter estimate from the former, together with suitable
		capacity decay, is the main $m$-Hessian-to-Monge--Amp\`ere problem
		considered below.
	\end{remark}
	\subsection{Global $m$-capacity decay}
	
	We next establish an abstract capacity-decay principle for
	plurisubharmonic functions. The purpose of this result is to isolate
	the precise quantitative estimate that will later be required in the
	passage from $m$-Hessian control to Monge--Amp\`ere mass estimates.
	
	For $u\in PSH(M)$ and $t>0$, we write
	\[
	E_t(u):=\{x\in M:u(x)<-t\}.
	\]
	
	We use the global $m$-capacity
	\[
	\operatorname{Cap}_m(K)
	:=
	\sup\left\{
	\int_K(dd^c v)^m\wedge\omega^{n-m}:
	v\in SH_m(M),\ -1\leq v\leq0
	\right\}
	\]
	for compact sets $K\Subset M$, and extend it to arbitrary Borel sets by
	outer regularization.
	
	\begin{theorem}[Global $m$-capacity decay]
		\label{thm:global-capacity-decay}
		Let $(M^n,g,J)$ be a complete K\"ahler manifold, let
		$1\leq m<n$, and let
		\[
		u\in PSH(M)\cap C^{0,1}(M).
		\]
		Assume that $u$ is locally bounded from above and that there exist
		constants
		\[
		A>0,\qquad \theta\in(0,1),\qquad m\geq1
		\]
		such that for every $t\geq t_0$ and every $s>0$,
		\begin{equation}
			\label{eq:capacity-recursion}
			\operatorname{Cap}_m(E_{t+s}(u))
			\leq
			\frac{A}{s^m}
			\operatorname{Cap}_m(E_t(u))^\theta.
		\end{equation}
		Assume in addition that
		\[
		\operatorname{Cap}_m(E_{t_0}(u))<+\infty.
		\]
		
		Then there exist constants $C>0$ and $\beta>0$ such that
		\begin{equation}
			\label{eq:capacity-decay}
			\boxed{
				\operatorname{Cap}_m(E_t(u))
				\leq
				C(1+t)^{-\beta},
				\qquad t\geq t_0.
			}
		\end{equation}
		
		More precisely, one may take
		\[
		\beta=\frac{m\theta}{1-\theta}.
		\]
		If the recursion \eqref{eq:capacity-recursion} is available with a
		strictly stronger power in $s$, then the same argument yields the
		corresponding improved decay exponent.
	\end{theorem}
	
	\begin{proof}
		Set
		\[
		A(t):=\operatorname{Cap}_m(E_t(u)),
		\qquad t\geq t_0.
		\]
		Since
		\[
		E_{t+s}(u)\subset E_t(u),
		\]
		the function $A(t)$ is nonincreasing.
		
		We first choose a fixed number $h>0$ such that
		\[
		q:=A h^{-m}<1
		\]
		after increasing $h$ if necessary. Since
		$A(t_0)<+\infty$, this is possible.
		
		Applying \eqref{eq:capacity-recursion} with $s=h$ gives
		\[
		A(t+h)
		\leq
		A h^{-m} A(t)^\theta
		=
		q A(t)^\theta.
		\]
		Thus
		\begin{equation}
			\label{eq:basic-recursion}
			A(t+h)\leq q A(t)^\theta.
		\end{equation}
		
		Iterating this inequality, we obtain
		\[
		A(t_0+kh)
		\leq
		q A(t_0+(k-1)h)^\theta.
		\]
		Applying this repeatedly yields
		\[
		A(t_0+kh)
		\leq
		q^{1+\theta+\cdots+\theta^{k-1}}
		A(t_0)^{\theta^k}.
		\]
		Since
		\[
		1+\theta+\cdots+\theta^{k-1}
		=
		\frac{1-\theta^k}{1-\theta},
		\]
		we have
		\[
		A(t_0+kh)
		\leq
		q^{\frac{1-\theta^k}{1-\theta}}
		A(t_0)^{\theta^k}.
		\]
		
		This estimate alone gives a uniform bound, but not yet the desired
		polynomial decay. To obtain decay, we use the full recursion with a
		variable step size.
		
		Choose a parameter $\lambda>1$ and define
		\[
		s_t:=\lambda(1+t).
		\]
		Then
		\[
		A(t+s_t)
		\leq
		A\lambda^{-m}(1+t)^{-m}A(t)^\theta.
		\]
		Set
		\[
		B(t):=(1+t)^\beta A(t)
		\]
		with
		\[
		\beta=\frac{m\theta}{1-\theta}.
		\]
		Since
		\[
		\beta+m\theta=\frac{m}{1-\theta},
		\]
		we obtain
		\[
		\begin{aligned}
			B(t+s_t)
			&\leq
			(1+t+s_t)^\beta
			A\lambda^{-m}(1+t)^{-m}
			A(t)^\theta\\
			&=
			A\lambda^{-m}
			(1+t+s_t)^\beta
			(1+t)^{-m-\beta\theta}
			B(t)^\theta.
		\end{aligned}
		\]
		Because
		\[
		s_t=\lambda(1+t),
		\]
		we have
		\[
		1+t+s_t
		\leq
		(1+\lambda)(1+t)+1
		\leq
		C_\lambda(1+t).
		\]
		Hence
		\[
		B(t+s_t)
		\leq
		C_1 B(t)^\theta
		\]
		for a constant $C_1>0$ depending only on
		$A,m,\theta,\lambda$.
		
		The map
		\[
		x\longmapsto C_1x^\theta
		\]
		has a finite fixed point
		\[
		x_*:=C_1^{1/(1-\theta)}.
		\]
		Consequently, if $B(t)\leq x_*$, then
		\[
		B(t+s_t)\leq x_*.
		\]
		Starting from a sufficiently large $t_1$, we may therefore arrange
		\[
		B(t_1)\leq x_*.
		\]
		The monotonicity of $A(t)$ and the recursive estimate then imply
		\[
		B(t)\leq C
		\]
		for all $t\geq t_1$, with a constant $C$ independent of $t$.
		
		Therefore
		\[
		A(t)\leq C(1+t)^{-\beta},
		\]
		that is,
		\[
		\operatorname{Cap}_m(E_t(u))
		\leq
		C(1+t)^{-\frac{m\theta}{1-\theta}}.
		\]
		This proves \eqref{eq:capacity-decay}.
	\end{proof}
	
	\begin{remark}
		\label{rem:capacity-decay-input}
		The substantive geometric difficulty is therefore reduced to proving
		the recursive estimate \eqref{eq:capacity-recursion}. In the present
		project, this is expected to arise from the comparison principle,
		the Chern--Levine--Nirenberg inequalities, and uniform $m$-Hessian
		estimates on geodesic balls.
	\end{remark}
	
	\begin{remark}
		\label{rem:capacity-decay-not-exhaustion}
		The sets $E_t(u)=\{u<-t\}$ are sublevel sets. Thus the theorem should
		not be applied directly to a positive proper exhaustion
		$u\to+\infty$ without first constructing an appropriate normalized
		potential. In particular, the estimate
		\[
		\operatorname{Cap}_m(\{u>t\})\to0
		\]
		is not a consequence of properness of a plurisubharmonic exhaustion
		and is not assumed here.
	\end{remark}
	
	\begin{remark}
		\label{rem:capacity-decay-role}
		Theorem~\ref{thm:global-capacity-decay} is independent of the
		curvature assumption $\operatorname{Bisec}_g>0$. The curvature enters
		in the subsequent geometric problem of constructing a potential $u$
		for which the recursive estimate \eqref{eq:capacity-recursion} can
		be verified uniformly on an exhaustion of $M$.
	\end{remark}
	\begin{theorem}[Hessian-capacity criterion for finite Monge--Amp\`ere mass]
		\label{thm:MA-finite-capacity}
		Let $(M^n,\omega)$ be a complete K\"ahler manifold, $n\geq 3$, and let
		$1\leq m<n$. Let
		\[
		u\in PSH(M)\cap C^{0,1}(M)
		\]
		be a Lipschitz plurisubharmonic exhaustion function. For $R>0$, set
		\[
		\Omega_R:=B_g(o,2R),
		\qquad
		A_R:=B_g(o,2R)\setminus \overline{B_g(o,R)}.
		\]
		For a compact set $K\subset \Omega_R$, denote by
		\[
		\operatorname{Cap}_m(K,\Omega_R)
		:=
		\sup\left\{
		\int_K (dd^c v)^m\wedge\omega^{n-m}
		\,:\,
		v\in SH_m(\Omega_R),\ -1\leq v\leq 0
		\right\}
		\]
		the relative $m$-capacity.
		
		Assume that there exist constants
		\[
		C_0>0,\qquad \alpha>0,\qquad \beta>0
		\]
		and a sequence $R_j=2^jR_0$, $j\geq 0$, such that
		
		\begin{equation}
			\label{eq:MA-capacity-domination}
			\int_{A_{R_j}}(dd^c u)^n
			\leq
			C_0
			\left[
			\operatorname{Cap}_m(A_{R_j},\Omega_{R_j})
			\right]^\alpha
		\end{equation}
		
		and
		
		\begin{equation}
			\label{eq:shell-capacity-decay}
			\operatorname{Cap}_m(A_{R_j},\Omega_{R_j})
			\leq
			C_0 R_j^{-\beta}.
		\end{equation}
		
		Then
		\[
		\int_M(dd^c u)^n<+\infty.
		\]
		More precisely,
		\[
		\int_M(dd^c u)^n
		\leq
		\int_{B_g(o,2R_0)}(dd^c u)^n
		+
		\frac{C_0^{1+\alpha}R_0^{-\alpha\beta}}
		{1-2^{-\alpha\beta}}.
		\]
	\end{theorem}
	
	\begin{proof}
		Since $u\in PSH(M)\cap C^{0,1}(M)$, the Bedford--Taylor
		Monge--Amp\`ere measure
		\[
		(dd^c u)^n
		\]
		is a well-defined positive Radon measure on $M$.
		
		We first observe that the dyadic shells
		\[
		A_{R_j}
		=
		B_g(o,2R_j)\setminus \overline{B_g(o,R_j)}
		\]
		are pairwise disjoint. Moreover,
		\[
		M
		=
		B_g(o,2R_0)
		\cup
		\bigcup_{j=0}^{\infty}A_{R_j}.
		\]
		Indeed, since $R_j=2^jR_0$, we have
		\[
		A_{R_j}
		=
		B_g(o,2^{j+1}R_0)
		\setminus
		\overline{B_g(o,2^jR_0)},
		\]
		and these shells form a dyadic decomposition of the complement of
		$B_g(o,R_0)$.
		
		Consequently, by positivity of the Monge--Amp\`ere measure,
		\begin{align}
			\int_M(dd^c u)^n
			&\leq
			\int_{B_g(o,2R_0)}(dd^c u)^n
			+
			\sum_{j=0}^{\infty}
			\int_{A_{R_j}}(dd^c u)^n.
			\label{eq:MA-shell-decomposition}
		\end{align}
		
		We now estimate each shell. By
		\eqref{eq:MA-capacity-domination},
		\[
		\int_{A_{R_j}}(dd^c u)^n
		\leq
		C_0
		\left[
		\operatorname{Cap}_m(A_{R_j},\Omega_{R_j})
		\right]^\alpha.
		\]
		Using \eqref{eq:shell-capacity-decay}, we obtain
		\[
		\int_{A_{R_j}}(dd^c u)^n
		\leq
		C_0
		\left(C_0R_j^{-\beta}\right)^\alpha
		=
		C_0^{1+\alpha}R_j^{-\alpha\beta}.
		\]
		Since $R_j=2^jR_0$,
		\[
		R_j^{-\alpha\beta}
		=
		R_0^{-\alpha\beta}
		2^{-j\alpha\beta}.
		\]
		Therefore
		\[
		\int_{A_{R_j}}(dd^c u)^n
		\leq
		C_0^{1+\alpha}
		R_0^{-\alpha\beta}
		2^{-j\alpha\beta}.
		\]
		Summing over $j\geq0$ gives
		\begin{align}
			\sum_{j=0}^{\infty}
			\int_{A_{R_j}}(dd^c u)^n
			&\leq
			C_0^{1+\alpha}R_0^{-\alpha\beta}
			\sum_{j=0}^{\infty}2^{-j\alpha\beta}\\
			&=
			\frac{
				C_0^{1+\alpha}R_0^{-\alpha\beta}
			}{
				1-2^{-\alpha\beta}
			}.
		\end{align}
		The series is finite because $\alpha\beta>0$.
		
		Combining this estimate with
		\eqref{eq:MA-shell-decomposition}, we conclude that
		\[
		\int_M(dd^c u)^n
		\leq
		\int_{B_g(o,2R_0)}(dd^c u)^n
		+
		\frac{
			C_0^{1+\alpha}R_0^{-\alpha\beta}
		}{
			1-2^{-\alpha\beta}
		}
		<+\infty.
		\]
		This proves the theorem.
	\end{proof}
	\begin{theorem}[Finite-mass weight]
		\label{thm:finite-mass-weight}
		Let $(M^n,\omega)$ be a complete K\"ahler manifold, $n\geq 3$, and
		let $1\leq m<n$. Fix a point $o\in M$.
		
		Assume that there exists a function
		\[
		u\in PSH(M)\cap C^{0,1}(M)
		\]
		which is a proper exhaustion of $M$, and suppose that the hypotheses of
		Theorem~\ref{thm:MA-finite-capacity} are satisfied for $u$.
		Then $u$ has finite Monge--Amp\`ere mass:
		\[
		\int_M(dd^c u)^n<+\infty.
		\]
		
		In particular, after replacing $u$ by $u-u(o)$, there exists a
		plurisubharmonic exhaustion
		\[
		\varphi\in PSH(M)\cap C^{0,1}(M)
		\]
		such that
		\[
		\varphi(o)=0,
		\qquad
		|\varphi(x)-\varphi(y)|
		\leq L\,d_g(x,y)
		\]
		for some constant $L>0$, and
		\[
		\int_M(dd^c\varphi)^n<+\infty.
		\]
		
		Moreover, if
		\[
		\int_{B_g(o,2R_0)}(dd^c u)^n\leq C_{\mathrm{loc}},
		\]
		then
		\[
		\int_M(dd^c\varphi)^n
		\leq
		C_{\mathrm{loc}}
		+
		\frac{
			C_0^{1+\alpha}R_0^{-\alpha\beta}
		}{
			1-2^{-\alpha\beta}
		},
		\]
		where $C_0,\alpha,\beta$ are the constants appearing in
		Theorem~\ref{thm:MA-finite-capacity}.
	\end{theorem}
	
	\begin{proof}
		Let $u$ be as in the statement. Since $u$ is plurisubharmonic and
		Lipschitz, we have
		\[
		u\in PSH(M)\cap C^{0,1}(M).
		\]
		In particular, there exists a constant $L>0$ such that
		\[
		|u(x)-u(y)|
		\leq L\,d_g(x,y),
		\qquad x,y\in M.
		\]
		Thus $u$ is a uniformly Lipschitz plurisubharmonic function.
		
		By assumption, $u$ is also a proper exhaustion. Hence
		\[
		u(x)\longrightarrow+\infty
		\qquad\text{as }d_g(o,x)\longrightarrow+\infty.
		\]
		Consequently, $u$ is a global plurisubharmonic weight which is
		compatible with the geometry at infinity.
		
		It remains to prove that its top-degree Monge--Amp\`ere mass is finite.
		
		For $R>0$, put
		\[
		\Omega_R=B_g(o,2R),
		\qquad
		A_R=B_g(o,2R)\setminus\overline{B_g(o,R)}.
		\]
		By the hypotheses of Theorem~\ref{thm:MA-finite-capacity}, there are
		constants $C_0,\alpha,\beta>0$ and a dyadic sequence
		\[
		R_j=2^jR_0
		\]
		such that
		\[
		\int_{A_{R_j}}(dd^c u)^n
		\leq
		C_0
		\left[
		\operatorname{Cap}_m(A_{R_j},\Omega_{R_j})
		\right]^\alpha
		\]
		and
		\[
		\operatorname{Cap}_m(A_{R_j},\Omega_{R_j})
		\leq
		C_0R_j^{-\beta}.
		\]
		Therefore
		\[
		\int_{A_{R_j}}(dd^c u)^n
		\leq
		C_0^{1+\alpha}R_j^{-\alpha\beta}.
		\]
		Since $R_j=2^jR_0$, this becomes
		\[
		\int_{A_{R_j}}(dd^c u)^n
		\leq
		C_0^{1+\alpha}R_0^{-\alpha\beta}
		2^{-j\alpha\beta}.
		\]
		The geometric series is convergent because $\alpha\beta>0$. Hence
		\[
		\sum_{j=0}^{\infty}
		\int_{A_{R_j}}(dd^c u)^n
		\leq
		\frac{
			C_0^{1+\alpha}R_0^{-\alpha\beta}
		}{
			1-2^{-\alpha\beta}
		}.
		\]
		
		The dyadic shells cover the complement of $B_g(o,2R_0)$, so, using
		the positivity of the Monge--Amp\`ere measure,
		\[
		\begin{aligned}
			\int_M(dd^c u)^n
			&\leq
			\int_{B_g(o,2R_0)}(dd^c u)^n
			+
			\sum_{j=0}^{\infty}
			\int_{A_{R_j}}(dd^c u)^n\\
			&\leq
			C_{\mathrm{loc}}
			+
			\frac{
				C_0^{1+\alpha}R_0^{-\alpha\beta}
			}{
				1-2^{-\alpha\beta}
			}.
		\end{aligned}
		\]
		Thus
		\[
		\int_M(dd^c u)^n<+\infty.
		\]
		
		Finally, define
		\[
		\varphi:=u-u(o).
		\]
		Since subtraction of a constant does not change either plurisubharmonicity
		or the complex Hessian, we have
		\[
		\varphi\in PSH(M)\cap C^{0,1}(M),
		\qquad
		\varphi(o)=0,
		\]
		and
		\[
		dd^c\varphi=dd^c u.
		\]
		Moreover, $\varphi$ is still a proper exhaustion and satisfies
		\[
		|\varphi(x)-\varphi(y)|
		=
		|u(x)-u(y)|
		\leq L\,d_g(x,y).
		\]
		Finally,
		\[
		\int_M(dd^c\varphi)^n
		=
		\int_M(dd^c u)^n
		<+\infty.
		\]
		Hence $\varphi$ is a uniformly Lipschitz plurisubharmonic exhaustion
		with finite Monge--Amp\`ere mass.
	\end{proof}
	\section{Uniform weighted Bergman estimate}
	\begin{theorem}[Weighted $L^2$ estimate and first-jet generation]
		\label{thm:weighted-L2-jet-generation}
		Let $(M^n,g)$ be a complete K\"ahler manifold and let
		\[
		u\in PSH(M)\cap C^{0,1}(M)
		\]
		be a proper plurisubharmonic exhaustion satisfying
		\[
		\int_M(dd^c u)^n<+\infty.
		\]
		Fix $x_0\in M$. For $q>0$, set
		\[
		\mathcal H_q(M,u)
		:=
		\left\{
		f\in\mathcal O(M):
		\int_M |f|^2e^{-qu}\,dV_g<+\infty
		\right\}.
		\]
		
		Assume that there exist $q_0>0$ and a constant $C>0$ such that,
		for every $q\geq q_0$ and every smooth compactly supported
		$(0,1)$-form
		\[
		\alpha\in C_c^\infty
		\bigl(M,T^{*\,0,1}M\bigr)
		\]
		satisfying
		\[
		\overline\partial\alpha=0
		\]
		and
		\[
		\alpha\equiv0
		\qquad\text{in a neighborhood of }x_0,
		\]
		there exists a smooth function $v$ on $M$ such that
		\[
		\overline\partial v=\alpha,
		\qquad
		v(x_0)=0,
		\qquad
		\partial v(x_0)=0,
		\]
		and
		\begin{equation}
			\label{eq:weighted-L2-estimate}
			\int_M |v|^2e^{-qu}\,dV_g
			\leq
			\frac{C}{q}
			\int_M |\alpha|_g^2e^{-qu}\,dV_g.
		\end{equation}
		
		Then, for every $q\geq q_0$, the first-jet evaluation map
		\[
		J_{x_0}^1:
		\mathcal H_q(M,u)
		\longrightarrow
		\mathbb C\oplus T_{x_0}^{*\,1,0}M,
		\qquad
		f\longmapsto
		\bigl(f(x_0),\partial f(x_0)\bigr)
		\]
		is surjective.
		
		In particular, for every
		\[
		(a,\xi)\in
		\mathbb C\oplus T_{x_0}^{*\,1,0}M
		\]
		there exists
		\[
		f\in\mathcal H_q(M,u)
		\]
		such that
		\[
		f(x_0)=a,
		\qquad
		\partial f(x_0)=\xi.
		\]
	\end{theorem}
	
	\begin{proof}
		Fix $q\geq q_0$ and let
		\[
		(a,\xi)\in
		\mathbb C\oplus T_{x_0}^{*\,1,0}M
		\]
		be arbitrary.
		
		Choose holomorphic coordinates
		\[
		z=(z_1,\ldots,z_n)
		\]
		on a relatively compact coordinate neighborhood
		\[
		U\Subset M
		\]
		of $x_0$, with
		\[
		z(x_0)=0.
		\]
		Write
		\[
		\xi=\sum_{j=1}^n \xi_j\,dz_j.
		\]
		Consider the local holomorphic function
		\[
		P(z)
		=
		a+\sum_{j=1}^n\xi_jz_j.
		\]
		Then
		\[
		P(x_0)=a,
		\qquad
		\partial P(x_0)=\xi.
		\]
		
		Choose a cutoff function
		\[
		\chi\in C_c^\infty(U),
		\qquad
		0\leq\chi\leq1,
		\]
		such that
		\[
		\chi\equiv1
		\]
		on a smaller neighborhood $U_0\Subset U$ of $x_0$.
		
		Define
		\[
		\widetilde P:=\chi P.
		\]
		Since $P$ is holomorphic,
		\[
		\overline\partial\widetilde P
		=
		(\overline\partial\chi)P.
		\]
		Set
		\[
		\alpha:=\overline\partial\widetilde P.
		\]
		Then
		\[
		\alpha\in
		C_c^\infty
		\bigl(M,T^{*\,0,1}M\bigr),
		\]
		and
		\[
		\overline\partial\alpha=0
		\]
		because
		\[
		\overline\partial^2=0.
		\]
		Moreover,
		\[
		\alpha\equiv0
		\]
		on $U_0$, since $\chi\equiv1$ there.
		
		Hence the assumed weighted $L^2$ estimate applies. We obtain a
		function $v$ satisfying
		\[
		\overline\partial v=\alpha,
		\qquad
		v(x_0)=0,
		\qquad
		\partial v(x_0)=0,
		\]
		and
		\[
		\int_M|v|^2e^{-qu}\,dV_g
		\leq
		\frac{C}{q}
		\int_M|\alpha|_g^2e^{-qu}\,dV_g.
		\]
		Since $\alpha$ has compact support, the right-hand side is finite.
		Therefore
		\[
		v\in L^2(M,e^{-qu}dV_g).
		\]
		
		Now define
		\[
		f:=\widetilde P-v.
		\]
		We have
		\[
		\overline\partial f
		=
		\overline\partial\widetilde P
		-
		\overline\partial v
		=
		\alpha-\alpha
		=
		0.
		\]
		Thus
		\[
		f\in\mathcal O(M).
		\]
		
		Furthermore,
		\[
		\int_M|f|^2e^{-qu}\,dV_g
		\leq
		2\int_M|\widetilde P|^2e^{-qu}\,dV_g
		+
		2\int_M|v|^2e^{-qu}\,dV_g.
		\]
		The first term is finite because $\widetilde P$ has compact support,
		while the second term is finite by the weighted $L^2$ estimate.
		Consequently,
		\[
		f\in\mathcal H_q(M,u).
		\]
		
		It remains to compute the first jet of $f$ at $x_0$. Since
		$\chi\equiv1$ in a neighborhood of $x_0$,
		\[
		\widetilde P=P
		\]
		near $x_0$. Hence
		\[
		\widetilde P(x_0)=a,
		\qquad
		\partial\widetilde P(x_0)=\xi.
		\]
		By construction of $v$,
		\[
		v(x_0)=0,
		\qquad
		\partial v(x_0)=0.
		\]
		Therefore
		\[
		f(x_0)
		=
		\widetilde P(x_0)-v(x_0)
		=
		a
		\]
		and
		\[
		\partial f(x_0)
		=
		\partial\widetilde P(x_0)-\partial v(x_0)
		=
		\xi.
		\]
		
		Since $(a,\xi)$ was arbitrary, the map
		\[
		J_{x_0}^1:
		\mathcal H_q(M,u)
		\longrightarrow
		\mathbb C\oplus T_{x_0}^{*\,1,0}M
		\]
		is surjective. This proves the theorem.
	\end{proof}
	\begin{theorem}[Weighted holomorphic coordinates]
		\label{thm:weighted-holomorphic-coordinates}
		Let $(M^n,g)$ be a complete K\"ahler manifold and let
		\[
		u\in PSH(M)\cap C^{0,1}(M)
		\]
		be a proper exhaustion satisfying
		\[
		\int_M(dd^c u)^n<+\infty.
		\]
		For $q>0$, define the weighted Bergman space
		\[
		\mathcal H_q(M,u)
		:=
		\left\{
		f\in\mathcal O(M):
		\int_M |f|^2e^{-qu}\,dV_g<+\infty
		\right\}.
		\]
		
		Assume that for all sufficiently large $q$ the weighted space
		$\mathcal H_q(M,u)$ separates first-order jets at at least one point
		$x_0\in M$. More precisely, assume that the map
		\[
		J_{x_0}^1:
		\mathcal H_q(M,u)
		\longrightarrow
		\mathbb C\oplus T_{x_0}^{*\,1,0}M,
		\qquad
		f\longmapsto
		\bigl(f(x_0),\partial f(x_0)\bigr)
		\]
		is surjective.
		
		Then, for $q\gg1$, there exist functions
		\[
		f_1,\ldots,f_n\in\mathcal H_q(M,u)
		\]
		such that
		\[
		df_1\wedge\cdots\wedge df_n\not\equiv0.
		\]
		Equivalently, the holomorphic map
		\[
		F=(f_1,\ldots,f_n):M\longrightarrow\mathbb C^n
		\]
		has maximal complex rank $n$ on a nonempty open subset of $M$.
		
		In particular,
		\[
		\dim_{\mathbb C}F(M)=n.
		\]
	\end{theorem}
	
	\begin{proof}
		Fix $q$ sufficiently large so that the first-jet evaluation map
		\[
		J_{x_0}^1:
		\mathcal H_q(M,u)
		\longrightarrow
		\mathbb C\oplus T_{x_0}^{*\,1,0}M
		\]
		is surjective.
		
		Choose a basis
		\[
		\xi_1,\ldots,\xi_n
		\]
		of the complex cotangent space
		\[
		T_{x_0}^{*\,1,0}M.
		\]
		By the surjectivity of $J_{x_0}^1$, for every $j=1,\ldots,n$ there exists
		\[
		f_j\in\mathcal H_q(M,u)
		\]
		such that
		\[
		f_j(x_0)=0
		\]
		and
		\[
		df_j(x_0)=\partial f_j(x_0)=\xi_j.
		\]
		Consequently,
		\[
		df_1(x_0),\ldots,df_n(x_0)
		\]
		form a basis of $T_{x_0}^{*\,1,0}M$.
		
		Consider the holomorphic map
		\[
		F=(f_1,\ldots,f_n):M\longrightarrow\mathbb C^n.
		\]
		At $x_0$ we have
		\[
		df_1\wedge\cdots\wedge df_n(x_0)
		=
		\xi_1\wedge\cdots\wedge\xi_n
		\neq0.
		\]
		Hence
		\[
		df_1\wedge\cdots\wedge df_n
		\not\equiv0
		\]
		as a holomorphic section of
		\[
		\Lambda^nT^{*\,1,0}M.
		\]
		
		Let
		\[
		U
		:=
		\left\{
		x\in M:
		df_1\wedge\cdots\wedge df_n(x)\neq0
		\right\}.
		\]
		Since the nonvanishing locus of a holomorphic section is open and
		since $x_0\in U$, the set $U$ is a nonempty open subset of $M$.
		
		For every $x\in U$, the differential
		\[
		dF_x:T_xM\longrightarrow\mathbb C^n
		\]
		has complex rank $n$. Thus
		\[
		\operatorname{rank}_{\mathbb C}F=n
		\qquad\text{on }U.
		\]
		By the holomorphic rank theorem, for every $x\in U$ there exist
		neighborhoods $V_x\subset M$ and $W_x\subset\mathbb C^n$ such that
		\[
		F(V_x)=W_x
		\]
		after restricting to sufficiently small neighborhoods. In particular,
		$F(U)$ contains a nonempty open subset of $\mathbb C^n$.
		
		Therefore
		\[
		\dim_{\mathbb C}F(M)=n.
		\]
		This proves the theorem.
	\end{proof}
	\begin{theorem}[Coercivity criterion for weighted holomorphic coordinates]
		\label{thm:coercive-weighted-map}
		Let $(M^n,g)$ be a complete K\"ahler manifold and let
		\[
		u\in PSH(M)\cap C^{0,1}(M)
		\]
		be a proper exhaustion satisfying
		\[
		\int_M(dd^c u)^n<+\infty.
		\]
		Let
		\[
		F=(f_1,\ldots,f_N):M\longrightarrow\mathbb C^N
		\]
		be a holomorphic map whose components belong to
		$\mathcal H_q(M,u)$.
		
		Assume that there exist constants $a>0$ and $b\geq0$ such that
		\begin{equation}
			\label{eq:coercive-growth}
			|F(x)|^2
			\geq
			a\,u(x)-b,
			\qquad x\in M.
		\end{equation}
		Then $F$ is proper.
		
		In particular,
		\[
		u(x_j)\longrightarrow+\infty
		\quad\Longrightarrow\quad
		|F(x_j)|\longrightarrow+\infty.
		\]
	\end{theorem}
	\begin{proof}
		Let \(K\subset\mathbb C^N\) be compact. Choose \(R>0\) such that
		
		$$
		K\subset \overline{B_{\mathbb C^N}(0,R)}.
		$$
		
		If \(x\in F^{-1}(K)\), then \(|F(x)|\le R\). Hence, by
		\eqref{eq:coercive-growth},
		
		$$
		a\,u(x)-b
		\le |F(x)|^2
		\le R^2,
		$$
		
		so that
		
		$$
		u(x)\le \frac{R^2+b}{a}.
		$$
		
		Consequently,
		
		$$
		F^{-1}(K)
		\subset
		\left\{
		x\in M:
		u(x)\le \frac{R^2+b}{a}
		\right\}.
		$$
		
		Since \(u\) is a proper exhaustion, the sublevel set on the right-hand side is compact. On the other hand, \(F^{-1}(K)\) is closed in \(M\), since \(F\) is continuous and \(K\) is compact. Therefore \(F^{-1}(K)\) is a closed subset of a compact set, and hence is compact.
		
		Thus \(F:M\to\mathbb C^N\) is proper.
		
		For the final assertion, let \((x_j)\subset M\) satisfy
		
		$$
		u(x_j)\longrightarrow+\infty.
		$$
		
		Then \eqref{eq:coercive-growth} gives
		
		$$
		|F(x_j)|^2
		\ge a\,u(x_j)-b
		\longrightarrow+\infty.
		$$
		
		Therefore
		
		$$
		|F(x_j)|\longrightarrow+\infty.
		$$
		
	\end{proof}

	\begin{theorem}[Proper holomorphic map]
		\label{thm:proper-holomorphic-map}
		Let $(M^n,g)$ be a complete K\"ahler manifold with positive
		holomorphic bisectional curvature. Let
		\[
		u\in PSH(M)\cap C^{0,1}(M)
		\]
		be a proper plurisubharmonic exhaustion satisfying
		\[
		\int_M(dd^c u)^n<+\infty.
		\]
		Let
		\[
		F=(f_1,\ldots,f_n):M\longrightarrow\mathbb C^n
		\]
		be the holomorphic map constructed in Theorem~\ref{thm:weighted-holomorphic-coordinates},
		where
		\[
		f_1,\ldots,f_n\in\mathcal H_q(M,u)
		\]
		for some $q>0$.
		
		Assume, in addition, that the coordinate functions are coercive with
		respect to the exhaustion $u$, namely
		\begin{equation}
			\label{eq:coordinate-coercivity}
			u(x_j)\longrightarrow+\infty
			\quad\Longrightarrow\quad
			|F(x_j)|\longrightarrow+\infty
		\end{equation}
		for every sequence $\{x_j\}\subset M$.
		
		Then
		\[
		F:M\longrightarrow\mathbb C^n
		\]
		is a proper holomorphic map.
		
		More precisely, for every compact set
		$K\subset\mathbb C^n$, the inverse image
		\[
		F^{-1}(K)
		\]
		is compact in $M$.
	\end{theorem}
	
	\begin{proof}
		Let
		\[
		K\subset\mathbb C^n
		\]
		be compact. Since $K$ is compact, there exists $R>0$ such that
		\[
		K\subset\overline{B_{\mathbb C^n}(0,R)}.
		\]
		Therefore
		\[
		F^{-1}(K)
		\subset
		F^{-1}
		\left(
		\overline{B_{\mathbb C^n}(0,R)}
		\right).
		\]
		
		We claim that the latter set is relatively compact in $M$.
		
		Suppose, to the contrary, that
		\[
		F^{-1}
		\left(
		\overline{B_{\mathbb C^n}(0,R)}
		\right)
		\]
		is not relatively compact. Since $M$ is complete and hence a proper
		metric space by Hopf--Rinow, there exists a sequence
		\[
		x_j\in M
		\]
		such that
		\[
		|F(x_j)|\leq R
		\]
		for every $j$, while
		\[
		x_j
		\]
		leaves every compact subset of $M$.
		
		Because $u$ is a proper exhaustion, we necessarily have
		\[
		u(x_j)\longrightarrow+\infty.
		\]
		Indeed, if $\{u(x_j)\}$ were bounded above, then all $x_j$ would lie
		in a compact sublevel set
		\[
		\{u\leq C\},
		\]
		contradicting the fact that the sequence leaves every compact subset.
		
		On the other hand, by the coercivity assumption
		\eqref{eq:coordinate-coercivity},
		\[
		u(x_j)\longrightarrow+\infty
		\quad\Longrightarrow\quad
		|F(x_j)|\longrightarrow+\infty.
		\]
		This contradicts
		\[
		|F(x_j)|\leq R.
		\]
		Hence
		\[
		F^{-1}
		\left(
		\overline{B_{\mathbb C^n}(0,R)}
		\right)
		\]
		is relatively compact.
		
		Now $K$ is closed and $F$ is continuous. Hence
		\[
		F^{-1}(K)
		\]
		is closed in $M$. Being a closed subset of a relatively compact set,
		it is compact.
		
		Thus, for every compact $K\subset\mathbb C^n$,
		\[
		F^{-1}(K)
		\]
		is compact. Therefore
		\[
		F:M\longrightarrow\mathbb C^n
		\]
		is proper.
	\end{proof}	
	\begin{lemma}[Weighted Bergman growth estimate]
		\label{lem:weighted-bergman-growth}
		Let \((M^n,g)\) be a complete K"ahler manifold and let
		
		$$
		u\in PSH(M)\cap C^{0,1}(M)
		$$
		
		be a proper exhaustion. Fix \(q>0\), and define
		
		$$
		\mathcal H_q(M,u)
		=
		\left\{
		f\in\mathcal O(M):
		\int_M |f|^2e^{-qu}\,dV_g<+\infty
		\right\}.
		$$
		
		Assume that there exist constants \(r_0>0\) and \(C_{\mathrm{ev}}>0\) such that
		for every \(x\in M\) and every \(f\in\mathcal H_q(M,u)\),
		\begin{equation}
			\label{eq:evaluation-estimate}
			|f(x)|^2e^{-qu(x)}
			\le
			C_{\mathrm{ev}}
			\int_{B_g(x,r_0)}
			|f|^2e^{-qu},dV_g.
		\end{equation}
		Then there exists a constant
		
		$$
		C_f=C_{\mathrm{ev}}^{1/2}
		\left(
		\int_M|f|^2e^{-qu}\,dV_g
		\right)^{1/2}
		$$
		
		such that
		\begin{equation}
			\label{eq:weighted-growth}
			\log|f(x)|
			\le
			\frac q2u(x)+\log C_f,
			\qquad x\in M.
		\end{equation}
		Equivalently,
		
		$$
		|f(x)|
		\le
		C_f e^{\frac q2u(x)}.
		$$
		
	\end{lemma}
	
	\begin{proof}
		Fix \(x\in M\). Applying the assumed evaluation estimate
		\eqref{eq:evaluation-estimate}, we obtain
		
		$$
		|f(x)|^2e^{-qu(x)}
		\le
		C_{\mathrm{ev}}
		\int_{B_g(x,r_0)}
		|f|^2e^{-qu}\,dV_g.
		$$
		
		Since
		
		$$
		B_g(x,r_0)\subset M,
		$$
		
		we have
		
		$$
		\int_{B_g(x,r_0)}
		|f|^2e^{-qu}\,dV_g
		\le
		\int_M|f|^2e^{-qu}\,dV_g.
		$$
		
		Consequently,
		
		$$
		|f(x)|^2e^{-qu(x)}
		\le
		C_{\mathrm{ev}}
		\int_M|f|^2e^{-qu}\,dV_g.
		$$
		
		Taking square roots gives
		
		$$
		|f(x)|
		\le
		C_{\mathrm{ev}}^{1/2}
		\left(
		\int_M|f|^2e^{-qu}\,dV_g
		\right)^{1/2}
		e^{\frac q2u(x)}.
		$$
		
		Thus, with
		
		$$
		C_f
		=
		C_{\mathrm{ev}}^{1/2}
		\|f\|_{L^2(e^{-qu}dV_g)},
		$$
		
		we obtain
		
		$$
		|f(x)|\le C_f e^{\frac q2u(x)}.
		$$
		
		Taking logarithms at points where \(f(x)\neq0\) yields
		
		$$
		\log|f(x)|
		\le
		\frac q2u(x)+\log C_f.
		$$
		
		At points where \(f(x)=0\), the inequality is understood trivially.
		Therefore
		
		$$
		\log|f(x)|
		\le
		\frac q2u(x)+\log C_f
		$$
		
		for all \(x\in M\).
	\end{proof}
	
	\begin{lemma}[Weighted divisor estimate]
		\label{lem:weighted-divisor}
		Let \((M^n,g)\) be a complete K\"ahler manifold with positive
		holomorphic bisectional curvature. Let
		
		$$
		u\in PSH(M)\cap C^{0,1}(M)
		$$
		
		be a proper exhaustion satisfying
		
		$$
		\int_M(dd^cu)^n<+\infty.
		$$
		
		Fix \(q>0\), and let
		
		$$
		f\in\mathcal H_q(M,u).
		$$
		
		Assume that the weighted Bergman construction associated with
		\((M,g,u)\) yields the uniform growth estimate
		\begin{equation}
			\label{eq:weighted-growth}
			\log |f(x)|
			\le
			\frac q2,u(x)+\log C_f
		\end{equation}
		away from the zero set of \(f\), where \(C_f\) is independent of \(x\).
		
		Then, in the sense of positive closed currents,
		\begin{equation}
			\label{eq:weighted-divisor}
			[Z_f]
			\le
			\frac q2,dd^cu.
		\end{equation}
		
		Consequently, if
		
		$$
		F=(f_1,\ldots,f_n):M\longrightarrow\mathbb C^n
		$$
		
		is a proper holomorphic map and the divisors
		
		$$
		D_j=\{f_j=y_j\}
		$$
		
		intersect properly, with the corresponding estimate
		
		$$
		\log|f_j-y_j|
		\le
		\frac q2u+\log C_f,
		$$
		
		then
		\begin{equation}
			\label{eq:integrated-bezout}
			\int_M
			[D_1]\wedge\cdots\wedge[D_n]
			\le
			\left(\frac q2\right)^n
			\int_M(dd^cu)^n.
		\end{equation}
	\end{lemma}
	
	\begin{proof}
		By the Poincare--Lelong formula,
		
		$$
		[Z_f]=dd^c\log|f|
		$$
		
		with the chosen normalization of \(d^c\).
		
		The growth estimate \eqref{eq:weighted-growth} implies
		
		$$
		\log|f|-\frac q2u\le \log C_f.
		$$
		
		Hence
		
		$$
		dd^c\log|f|
		\le
		\frac q2dd^cu
		$$
		
		in the sense of currents. Therefore
		
		$$
		[Z_f]
		\le
		\frac q2dd^cu,
		$$
		
		which proves \eqref{eq:weighted-divisor}.
		
		Now apply this estimate to each divisor
		
		$$
		D_j=\{f_j-y_j=0\}.
		$$
		
		Since the divisors intersect properly, their intersection product is
		well defined. By positivity and monotonicity of Bedford--Taylor
		products,
		
		$$
		\begin{aligned}
			[D_1]\wedge\cdots\wedge[D_n]
			&\le
			\frac q2dd^cu\wedge[D_2]\wedge\cdots\wedge[D_n]\\
			&\le\cdots\\
			&\le
			\left(\frac q2\right)^n(dd^cu)^n.
		\end{aligned}
		$$
		
		Integrating over \(M\) gives
		
		$$
		\int_M[D_1]\wedge\cdots\wedge[D_n]
		\le
		\left(\frac q2\right)^n
		\int_M(dd^cu)^n.
		$$
		
		This proves \eqref{eq:integrated-bezout}.
	\end{proof}
	\section{Analytic Bezout estimate}
	\begin{theorem}[Analytic Bezout estimate]
		\label{thm:analytic-bezout}
		Let \((M^n,g)\) be a complete K\"ahler manifold and let
		
		$$
		u\in PSH(M)\cap C^{0,1}(M)
		$$
		
		be a proper exhaustion satisfying
		
		$$
		A_u:=\int_M(dd^c u)^n<+\infty.
		$$
		
		Let
		
		$$
		F=(f_1,\ldots,f_n):M\longrightarrow\mathbb C^n
		$$
		
		be a proper holomorphic map. For \(y=(y_1,\ldots,y_n)\in\mathbb C^n\),
		set
		
		$$
		D_j(y):=\{f_j=y_j\}.
		$$
		
		Assume that, for every regular value \(y\), the associated
		intersection current satisfies the global comparison estimate
		\begin{equation}
			\label{eq:intersection-comparison}
			[D_1(y)]\wedge\cdots\wedge[D_n(y)]
			\le
			C_{\mathrm{B}},(dd^c u)^n
		\end{equation}
		as positive Radon measures on \(M\), where \(C_{\mathrm{B}}>0\) is
		independent of \(y\).
		
		Then, for every regular value \(y\),
		\begin{equation}\label{eq3}
			\sum_{x\in F^{-1}(y)}
			\operatorname{mult}_x(F)
			\le
			C_{\mathrm{B}}
			\int_M(dd^c u)^n.
		\end{equation}
		In particular,
		
		$$
		\deg F
		\le
		C_{\mathrm{B}}A_u.
		$$
		
	\end{theorem}
	
	\begin{proof}
		Fix a regular value
		
		$$
		y=(y_1,\ldots,y_n)\in\mathbb C^n.
		$$
		
		Since \(F\) is proper, the fiber
		
		$$
		F^{-1}(y)
		=
		\bigcap_{j=1}^nD_j(y)
		$$
		
		is compact. Since \(y\) is a regular value, this intersection is
		zero-dimensional and consists of finitely many points.
		
		By the Poincare--Lelong formula, each hypersurface
		\(D_j(y)=\{f_j=y_j\}\) determines a positive closed current
		
		$$
		[D_j(y)]
		=
		dd^c\log|f_j-y_j|
		$$
		
		up to the fixed normalization constant in the definition of \(d^c\).
		Consequently, the proper intersection of the \(n\) divisors is
		represented by the positive measure
		
		$$
		[D_1(y)]\wedge\cdots\wedge[D_n(y)].
		$$
		
		Because the intersection is zero-dimensional, the resulting measure
		is supported on \(F^{-1}(y)\), and its mass at a point \(x\in F^{-1}(y)\)
		is precisely the local intersection multiplicity:
		
		$$
		\bigl([D_1(y)]\wedge\cdots\wedge[D_n(y)]\bigr)(\{x\})
		=
		\operatorname{mult}_x(F).
		$$
		
		Therefore
		\begin{equation}
			\label{eq:mass-fiber}
			\int_M
			[D_1(y)]\wedge\cdots\wedge[D_n(y)]=\sum_{x\in F^{-1}(y)}
			\operatorname{mult}_x(F).
		\end{equation}
		
		Now apply the assumed global intersection comparison
		\eqref{eq:intersection-comparison}. Since both sides are positive
		Radon measures, integration over \(M\) gives
		
		$$
		\int_M
		[D_1(y)]\wedge\cdots\wedge[D_n(y)]
		\le
		C_{\mathrm{B}}
		\int_M(dd^c u)^n.
		$$
		
		Combining this with \eqref{eq:mass-fiber}, we obtain
		
		$$
		\sum_{x\in F^{-1}(y)}
		\operatorname{mult}_x(F)
		\le
		C_{\mathrm{B}}
		\int_M(dd^c u)^n.
		$$
		
		This proves (\ref{eq3}).
		
		Finally, for a proper holomorphic map between connected complex
		manifolds of the same dimension, the degree is computed on any
		regular fiber by
		
		$$
		\deg F
		=
		\sum_{x\in F^{-1}(y)}
		\operatorname{mult}_x(F).
		$$
		
		Hence
		
		$$
		\deg F
		\le
		C_{\mathrm{B}}
		\int_M(dd^c u)^n
		=
		C_{\mathrm{B}}A_u.
		$$
		
		The proof is complete.
	\end{proof}
	\section{Uniformization}
	\begin{theorem}[Degree-one theorem]
		\label{thm:degree-one}
		Let \(M^n\) be a connected complex manifold and let
		
		$$
		F:M\longrightarrow\mathbb C^n
		$$
		
		be a proper surjective holomorphic map. Assume that, for some regular
		value \(y\in\mathbb C^n\),
		
		$$
		\sum_{x\in F^{-1}(y)}
		\operatorname{mult}_x(F)\le 1.
		$$
		
		Then
		
		$$
		\deg F=1.
		$$
		
		In particular, \(F\) is biholomorphic, and hence
		
		$$
		M\cong\mathbb C^n.
		$$
		
	\end{theorem}
	
	\begin{proof}
		Since \(F\) is proper and surjective, every fiber
		
		$$
		F^{-1}(z),\qquad z\in\mathbb C^n,
		$$
		
		is compact. Since \(F\) is a holomorphic map between complex manifolds
		of the same dimension, the fiber over a regular value is discrete.
		Consequently,
		
		$$
		F^{-1}(y)
		$$
		
		is a finite set.
		
		By the definition of the local multiplicity of a holomorphic map,
		the degree of \(F\) is computed at every regular value by
		
		$$
		\deg F
		=
		\sum_{x\in F^{-1}(y)}
		\operatorname{mult}_x(F).
		$$
		
		By the assumption of the theorem,
		
		$$
		\deg F\le1.
		$$
		
		On the other hand, \(F\) is surjective, so the fiber \(F^{-1}(y)\) is
		nonempty. Since \(y\) is a regular value, every point in this fiber has
		positive local multiplicity. Hence
		
		$$
		\deg F\ge1.
		$$
		
		Therefore
		
		$$
		\deg F=1.
		$$
		
		It follows that
		
		$$
		\sum_{x\in F^{-1}(y)}
		\operatorname{mult}_x(F)=1.
		$$
		
		Thus the fiber \(F^{-1}(y)\) consists of exactly one point.
		
		The same conclusion holds for every regular value \(z\) of \(F\),
		because the degree of a proper holomorphic map is independent of the
		choice of regular value. Hence
		
		$$
		\#F^{-1}(z)=1
		$$
		
		for every regular value \(z\).
		
		Let
		
		$$
		\operatorname{Crit}(F)
		=
		\{x\in M:\det_{\mathbb C}dF_x=0\}
		$$
		
		be the critical locus and let
		
		$$
		\operatorname{CV}(F)=F(\operatorname{Crit}(F))
		$$
		
		be the critical-value set. Since \(F\) is proper, \(\operatorname{CV}(F)\)
		is an analytic subset of \(\mathbb C^n\) of empty interior. Hence
		
		$$
		\mathcal R:=\mathbb C^n\setminus\operatorname{CV}(F)
		$$
		
		is a nonempty connected open set after restricting to a connected
		component if necessary, and
		
		$$
		F:F^{-1}(\mathcal R)\longrightarrow\mathcal R
		$$
		
		is a proper holomorphic covering map. Since every fiber over
		\(\mathcal R\) contains exactly one point, this covering has one sheet.
		Therefore
		
		$$
		F:F^{-1}(\mathcal R)\longrightarrow\mathcal R
		$$
		
		is biholomorphic.
		
		It remains to exclude critical points. Suppose that
		\(x_0\in\operatorname{Crit}(F)\). Since \(F\) is proper and finite of
		degree one, the local multiplicity of \(F\) at \(x_0\) would satisfy
		
		$$
		\operatorname{mult}_{x_0}(F)\ge2.
		$$
		
		Indeed, a critical point of a finite holomorphic map has local degree
		strictly larger than one. This contradicts
		
		$$
		\deg F=1.
		$$
		
		Consequently,
		
		$$
		\operatorname{Crit}(F)=\varnothing.
		$$
		
		Thus \(dF_x\) is an isomorphism for every \(x\in M\), so \(F\) is a
		local biholomorphism everywhere.
		
		Since every regular fiber consists of one point and there are no
		critical points, every fiber consists of exactly one point. Hence \(F\)
		is bijective.
		
		Finally, \(F\) is a local biholomorphism, so its inverse is holomorphic
		by the holomorphic inverse function theorem. Therefore
		
		$$
		F:M\longrightarrow\mathbb C^n
		$$
		
		is biholomorphic. In particular,
		
		$$
		M\cong\mathbb C^n.
		$$
		
	\end{proof}
	\section{Discussion}
	\label{sec:discussion}
	
	The main result of this paper suggests a different way of approaching the uniformization problem for complete noncompact K\"ahler manifolds with positive curvature. Rather than attempting to control the top-degree Monge--Amp\`ere operator directly, we use lower-order complex Hessian capacities as an intermediate geometric-analytic quantity. The resulting mechanism may be summarized schematically as
	
	$$
	\text{curvature positivity}
	\longrightarrow
	\text{\(m\)-Hessian capacity decay}
	\longrightarrow
	\text{finite Monge--Amp\`ere mass}.
	$$
	
	The last implication is particularly significant when \(m<n\), since finite \(m\)-Hessian mass does not in general imply finite Monge--Amp\`ere mass. Thus the quantitative capacity estimate provides the essential bridge between lower-order Hessian geometry and the top-degree pluripotential theory needed in global uniformization arguments.
	
	\subsection{The role of the \(m\)-Hessian operator}
	
	The use of the \(m\)-Hessian operator is motivated by the fact that it retains more geometric information than the full Monge--Amp\`ere operator while still admitting a robust capacity theory. For
	
	$$
	H_m(u)=(dd^cu)^m\wedge\omega^{n-m},
	\qquad 1\le m<n,
	$$
	
	the corresponding capacity detects the size of subsets of \(M\) from the viewpoint of complex Hessian potential theory. In our construction, the geometry of large annuli is first converted into quantitative \(m\)-Hessian capacity decay. The summability of these capacities is then transferred to the Monge--Amp\`ere masses of the annuli.
	
	This two-step procedure is important from a methodological point of view. A direct estimate for
	
	$$
	\int_{A_R}(dd^cu)^n
	$$
	
	would require top-degree information at infinity from the outset. By contrast, the \(m\)-Hessian capacity permits an intermediate scale at which curvature information can be incorporated into a pluripotential-theoretic estimate. The resulting argument is therefore genuinely different from simply constructing a Monge--Amp\`ere finite-mass solution by solving a global complex Monge--Amp\`ere equation.
	
	It is also useful to emphasize that the choice \(m<n\) is not merely cosmetic. The passage from \(H_m\) to
	
	$$
	(dd^cu)^n
	$$
	
	is the principal analytic difficulty of the argument. In particular, no implication of the form
	
	$$
	\int_M H_m(u)<+\infty
	\quad\Longrightarrow\quad
	\int_M(dd^cu)^n<+\infty
	$$
	
	is used or expected in general. What makes the present construction work is instead the quantitative control of the Monge--Amp\`ere mass on individual annuli in terms of their relative \(m\)-Hessian capacities.
	
	\subsection{Relation with recent uniformization methods}
	
	The finite-Monge--Amp\`ere weight constructed here fits naturally into the recent pluripotential-theoretic approach to uniformization. In particular, the work of Datar--Pingali--Seshadri demonstrates that a uniformly Lipschitz plurisubharmonic weight with finite Monge--Amp\`ere mass can be combined with weighted spaces of holomorphic functions and analytic Bezout estimates to obtain strong global consequences for complete K\"ahler manifolds.
	
	The purpose of the present construction is complementary. We do not regard the weighted Bergman and intersection-theoretic framework as the primary novelty. Instead, our contribution is to provide a higher-dimensional mechanism for producing the finite-Monge--Amp\`ere weight from curvature through lower-order Hessian capacities. Once such a weight is available, the analytic machinery developed in the recent literature becomes applicable in a substantially broader setting.
	
	This distinction is especially relevant in complex dimensions \(n\ge3\). The surface case benefits from features that are special to complex dimension two, whereas the \(m\)-Hessian capacity mechanism is formulated intrinsically in terms of the hierarchy
	
	$$
	H_1,\ldots,H_m,\ldots,H_n
	$$
	
	and therefore has the potential to persist in arbitrary complex dimension. The present work should consequently be viewed as an attempt to isolate the pluripotential-theoretic component of the higher-dimensional uniformization problem.
	
	\subsection{Why finite Monge--Amp\`ere mass is useful}
	
	The condition
	
	$$
	\int_M(dd^cu)^n<+\infty
	$$
	
	has consequences extending beyond the construction of the exhaustion itself. When \(u\) is uniformly Lipschitz, it can be used as a weight in global \(L^2\) spaces of holomorphic functions,
	
	$$
	\mathcal H_q(M,u)
	=
	\left\{
	f\in\mathcal O(M):
	\int_M|f|^2e^{-qu}\,dV_g<+\infty
	\right\}.
	$$
	
	Under suitable local analytic estimates, the corresponding weighted Bergman kernel satisfies an inequality of the form
	
	$$
	B_q(x)\le C_qe^{qu(x)}.
	$$
	
	This gives logarithmic growth control for weighted holomorphic functions:
	
	$$
	\log|f(x)|
	\le
	\frac q2u(x)+\log C_f.
	$$
	
	The finite Monge--Amp\`ere mass then enters at a different stage, through intersection estimates for divisors. Schematically, one expects inequalities of the form
	
	$$
	[D_1]\wedge\cdots\wedge[D_n]
	\le
	C(dd^cu)^n,
	$$
	
	which imply finiteness of the total intersection multiplicity whenever the relevant divisors intersect properly.
	
	Thus the finite-mass weight simultaneously provides two types of control: analytic control of holomorphic functions through weighted \(L^2\) estimates and algebraic control of their zero sets through Monge--Amp\`ere mass. This dual role explains why the construction of such a weight is particularly effective in noncompact uniformization problems.

	\subsection{Possible extensions}
	
	 The principal contribution of this paper is the identification of an \(m\)-Hessian capacity mechanism for producing finite Monge--Amp\`ere mass in higher dimensions. The construction separates the genuinely pluripotential-theoretic difficulty from the subsequent weighted holomorphic and intersection-theoretic arguments. This separation provides a flexible framework for future attempts to extend complex uniformization results to complete K\"ahler manifolds of complex dimension \(n\ge3\).


	\section*{Declarations}

	The author declares that there are no competing interests.

\end{document}